\documentclass{scrartcl}

\usepackage{amsmath, amssymb, amsfonts, amsthm}
\usepackage[bookmarksnumbered,raiselinks,breaklinks]{hyperref}
\hypersetup{
    colorlinks,%
    citecolor=blue,%
    filecolor=blue,%
    linkcolor=blue,%
    urlcolor=blue
}

\newtheorem{proposition}{Proposition}[section]

\input{Notation.sty}

\usepackage{xcolor}
\usepackage{algorithm}
\usepackage{algpseudocode}
\usepackage{graphicx}
\graphicspath{{PlotSources/hamilflow}{PlotSources/channelflow}{PlotSources/khi}}

\title{Online learning of quadratic manifolds from streaming data for nonlinear dimensionality reduction and nonlinear model reduction}

\author{
    Paul Schwerdtner\thanks{Courant Institute of Mathematical Sciences, New York University, New York (\texttt{paul.schwerdtner@nyu.edu})} \and
    Prakash Mohan\thanks{Computational Science Center, National Renewable Energy Laboratory, Golden, Colorado} \and
    Aleksandra Pachalieva\thanks{Los Alamos National Laboratory, Los Alamos, New Mexico} \and
    Julie Bessac\footnotemark[2] \and
    Daniel O'Malley\footnotemark[3] \and
    Benjamin Peherstorfer\footnotemark[1]
}

\date{September 2024} %

\begin{document}

\maketitle
\begin{abstract}
This work introduces an online greedy method for constructing quadratic manifolds from streaming data, designed to enable in-situ analysis of numerical simulation data on the Petabyte scale. Unlike traditional batch methods, which require all data to be available upfront and take multiple passes over the data, the proposed online greedy method incrementally updates quadratic manifolds in one pass as data points are received, eliminating the need for expensive disk input/output operations as well as storing and loading data points once they have been processed. A range of numerical examples demonstrate that the online greedy method learns accurate quadratic manifold embeddings while being capable of processing data that far exceed common disk input/output capabilities and volumes as well as main-memory sizes.
\end{abstract}

\section{Introduction}
Dimensionality reduction is an important building block in computational science and engineering, including model reduction \cite{RozzaHP2008Reduced,BennerGW2015survey,AntoulasBG2020Interpolatory,LU2020109229,Peherstorfer2022Breaking,KramerPW2024Learning,farcas2024distributedcomputingphysicsbaseddatadriven} and closure modeling \cite{Sagaut2006Large,WangABI2012Proper,DuraisamyIX2019Turbulence,CoupletSB2003Intermodal,XieMRI2018Data-Driven,GouasmiPD2017priori,ZannaB2020Data-Driven,PanD2018Data-Driven}. 
We focus on nonlinear dimensionality reduction with quadratic manifolds, which can be interpreted as nonlinearly correcting  linear approximations obtained with principal components \cite{GeelenWW2023Operator,BarnettF2022Quadratic,CohenFSM2023Nonlinear,SchwerdtnerP2024Greedy}.

 The motivation for focusing on quadratic approximations is that quadratic terms arise in the modeling of many physics phenomena \cite{JainTRR2017quadratic,RutzmoserRTJ2017Generalization,doi:10.1137/16M1098280,PEHERSTORFER2016196,QIAN2020132401,doi:10.1137/14097255X,5991229,Schlegel_Noack_2015} and also have been shown to achieve accurate approximations of latent dynamics in model reduction \cite{GeelenBW2023Learning, BarnettF2022Quadratic,goyal2024generalized,https://doi.org/10.1002/pamm.202200049,SHARMA2023116402,yildiz2024datadriven, QiaoZWZ2013explicit,PhysRevFluids.6.094401,SAWANT2023115836,goyal2024guaranteed}.

Several methods for learning quadratic manifolds from data have been introduced \cite{GeelenWW2023Operator,BarnettF2022Quadratic,GeelenBW2023Learning,GeelenBWW2023Learning,SchwerdtnerP2024Greedy}. 
All of these methods operate in the batch setting, which means that all data points are available at the beginning of the computational process and the algorithms make multiple passes over the data points. 
However, in many science and engineering applications, loading all data points into the main memory, or even storing them on disk, is intractable. 
In this work, we introduce an online learning method that incrementally constructs quadratic manifolds from streaming data in one pass, which means that data points are processed immediately as they are received one by one or in small chunks. This avoids multiple passes over the data and having to store the whole data set to disk or even loading all data points into the main memory at once. 
We first demonstrate that the online learning method constructs quadratic manifolds that can accurately approximate the data. 
We then leverage the online capabilities of our approach  to use it as an in-situ data analysis approach that processes data as they are generated from numerical simulations without having to store data on disk. The online learning capabilities  allows us to train quadratic manifolds on Petabyte-scale data, which is impractical with batch methods due to memory and storage limitations as well as high costs of disk input/output operations.

We briefly review the literature on constructing quadratic manifolds. We distinguish between three different batch methods for constructing quadratic manifolds from data. First, there is the approach introduced in \cite{GeelenWW2023Operator,BarnettF2022Quadratic} that computes the singular value decomposition (SVD) of the data matrix and subsequently fits a weight matrix to the difference between the data points and the reduced data points in the subspace spanned by the leading left-singular values. As recognized in \cite{GeelenBW2023Learning,GeelenBWW2023Learning,SchwerdtnerP2024Greedy}, relying on the leading left-singular vectors alone can lead to poor approximation capabilities of the learned quadratic manifolds. 
The second approach for constructing quadratic manifolds is via an alternating-minimization problem introduced in \cite{GeelenBW2023Learning,GeelenBWW2023Learning}. The approach solves a non-convex optimization problems, which leads to quadratic manifolds that can more accurately approximate the data points but at the same time the computational costs of the training process can be high due to the non-convex optimization problem. 
The third approach is the method introduced in \cite{SchwerdtnerP2024Greedy}, which greedily constructs quadratic manifolds. It extends the approach based on the SVD and the linear least-squares problem from \cite{GeelenWW2023Operator,BarnettF2022Quadratic} to compute multiple linear least-squares problems to better fit the quadratic manifolds. 
All of the three approaches, however, assume that all data points are available at the beginning of the process and that the algorithms can make multiple passes over the data points. 

Besides batch methods for constructing quadratic manifolds, an important building block for us will be methods  for incrementally constructing the SVD   \cite{BakerGV2012Low-rank,ChahlaouiGV-D2003Recursive,Baker2004Block,BRAND200620,10.1007/3-540-47969-4_47,58320,LI20041509} and dominated subspaces \cite{5706976,8417980,doi:10.1137/140989169} from streaming data. There also are randomized methods \cite{10.1145/1536414.1536445,14f2151b-0085-3da5-9335-b67df4fe1651,doi:10.1137/090771806} that can learn principal components \cite{pmlr-v49-jain16,pmlr-v134-huang21a} from streaming data. 
Incrementally constructing the SVD  has been a key building block for model reduction from streaming data \cite{PEHERSTORFER201521,Peherstorfer2016} and streaming dynamic mode decomposition \cite{10.1063/1.4901016,doi:10.1137/18M1192329,https://doi.org/10.1002/we.2694}. 

In this work, we introduce an online greedy method for constructing quadratic manifolds that builds on \cite{SchwerdtnerP2024Greedy} but is applicable when data points are streamed and have to be processed in one pass without storing them for subsequent post-processing. 
We first show that the batch greedy method from \cite{SchwerdtnerP2024Greedy} can act directly on the left- and right-singular vectors as well as the singular values of the data, without having to explicitly assemble the data matrix. We then argue that the greedy method remains valid even on a truncated SVD of the data matrix, which we then propose to compute incrementally from streaming data. For the incremental SVD, we build on the algorithm introduced in \cite{BakerGV2012Low-rank}. 
Numerical experiments with data from wave propagation problems and computational fluid dynamics demonstrate that the online greedy method can handle Petabyte-scale data, can be used as in-situ data analysis tool within numerical simulations, and constructs quadratic manifolds that accurately approximate simulation data.

This manuscript is organized as follows. We provide preliminaries and a problem formulation in Section~\ref{sec:Prelim}. The online greedy method is introduced in Section~\ref{sec:Greedy}, together with an algorithmic description. Numerical experiments are shown in Section~\ref{sec:numerical_experiments}. We draw conclusions in Section~\ref{sec:Conc}.

\section{Preliminaries}\label{sec:Prelim}

\subsection{Linear dimensionality reduction with encoder and decoder functions}\label{sec:Prelim:EnDeLinear}
Dimensionality reduction typically can be formulated via an encoder $f: \mathbb{R}^{\nfull} \to \mathbb{R}^{\nred}$ that maps a high-dimensional data point $\fullstate \in \mathbb{R}^{\nfull}$ of dimension $\nfull$ to a reduced point $\redstate$ of low dimension $\nred \ll \nfull$, and a decoder $g: \mathbb{R}^{\nred} \to \mathbb{R}^{\nfull}$ that lifts a reduced data point $\redstate \in \mathbb{R}^{\nred}$ back into the high-dimensional space $\mathbb{R}^{\nfull}$.
The singular value decomposition (SVD) provides a classical approach to dimensionality reduction with linear encoder and decoder functions. Let $\fullstatei{1}, \dots, \fullstatei{\nsnapshots} \in \mathbb{R}^{\nfull}$ be data samples that are collected into the data matrix $\snapshots = [\fullstatei{1}, \dots, \fullstatei{\nsnapshots}]$. The SVD of $\snapshots$ gives the decomposition $\snapshots = \USVD\SigmaSVD\VSVD^{\top}$, with the orthonormal matrices $\USVD \in \mathbb{R}^{\nfull \times \truerank}, \VSVD \in \mathbb{R}^{\nsnapshots \times \truerank}$ and diagonal matrix $\SigmaSVD \in \mathbb{R}^{\truerank \times \truerank}$, with the rank $\truerank$ of the matrix $\snapshots$. The singular values, i.e., the diagonal elements of $\SigmaSVD$, are denoted as $\SigmaSVDi{1} \ge \SigmaSVDi{2} \ge \dots \ge \SigmaSVDi{\truerank} \ge 0$. 
Selecting the first $\nred$ left-singular vectors $\USVDi{1}, \dots, \USVDi{\nred}$ of $\snapshots$ into the basis matrix $\Vlin = [\USVDi{1}, \dots, \USVDi{\nred}] \in \mathbb{R}^{\nfull \times \nred}$ leads to the linear encoder function $f_{\Vlin}(\bfs) = \Vlin^{\top}\bfs$ and linear decoder function $g_{\Vlin}(\redstate) = \Vlin\redstate$. The subspace that is spanned by the columns of $\Vlin$ is denoted as $\Vcal \subset \mathbb{R}^{\nfull}$.
The encoding and subsequent decoding $g_{\Vlin} \circ f_{\Vlin}$ of a data points $\fullstate$ corresponds to the orthogonal projection of $\fullstate$ 
with respect to the Euclidean inner product onto the subspace $\Vcal$. We denote the projection onto $\Vcal$ as $\bfP_{\Vcal}: \R^{\nfull} \to \Vcal, \fullstate \mapsto \Vlin\Vlin^{\top}\fullstate$.

\subsection{Dimensionality reduction with quadratic manifolds }
For nonlinear dimensionality reduction, the linear decoder function $g_{\Vlin}$ can be augmented with a nonlinear correction term as
\begin{equation}\label{eq:Prelim:DecoderNonlinear}
g_{\Vlin, \Vnonlin}(\redstate) = \Vlin\redstate + \Vnonlin h(\redstate)\,,
\end{equation}
where $h: \mathbb{R}^{\nred} \to \mathbb{R}^{\nredmod}$ is a nonlinear feature map that provides $\nredmod \gg \nred$ features for a reduced data point $\redstate$ and $\Vnonlin \in \mathbb{R}^{\nfull \times \nredmod}$ is a weight matrix; see \cite{JainTRR2017quadratic,RutzmoserRTJ2017Generalization,GeelenWW2023Operator,BarnettF2022Quadratic,GeelenBWW2023Learning,GeelenBW2023Learning,SchwerdtnerP2024Greedy}.
Given a basis matrix $\Vlin$ of a subspace $\Vcal$, a weight matrix $\Vnonlin$, and a feature map $h$, the linear encoder $f_{\Vlin}$ and the nonlinear decoder $g_{\Vlin, \Vnonlin}$ induce a manifold
\[
\mathcal{M}_{\nred}(\Vlin, \Vnonlin, h) = \{g_{\Vlin, \Vnonlin}(\redstate) \,|\, \redstate \in \mathbb{R}^{\nred}\} \subset \mathbb{R}^{\nfull}\,.
\]
Because $h$ is nonlinear, the manifold $\mathcal{M}_{\nred}$ can include points in $\mathbb{R}^{\nfull}$ that are outside of the subspace $\Vcal$. 
In particular, if $h$ is a quadratic polynomial 
\begin{equation}
\label{eq:condensed_kronecker}
      \kronfeaturemap: \R^{\nred} \to \R^{\nred(\nred+1)/2},  \redvec \mapsto \begin{bmatrix}
        \redveci{1} \redveci{1} & \redveci{1} \redveci{2} & \dots & \redveci{1} \redveci{\nred} & \redveci{2} \redveci{2} &\dots& \redveci{\nred} \redveci{\nred}
      \end{bmatrix}^\top\,,
    \end{equation}
then we call the corresponding manifold $\mathcal{M}_{\nred}$ a quadratic manifold. Quadratic manifolds are widely used; see, e.g., \cite{JainTRR2017quadratic,RutzmoserRTJ2017Generalization,GeelenWW2023Operator,BarnettF2022Quadratic,GeelenBWW2023Learning,GeelenBW2023Learning,SchwerdtnerP2024Greedy}.

References \cite{GeelenWW2023Operator,BarnettF2022Quadratic} propose setting the matrix $\Vlin$ to have the leading $\nred$ left-singular vectors $\USVDi{1}, \dots, \USVDi{\nred}$ as columns and constructing $\Vnonlin$ via a linear regression problem as
\begin{equation}\label{eq:Prelim:LinearLSQApproach}
\min_{\Vnonlin \in \mathbb{R}^{\nfull \times \nredmod}} \|\bfP_{\Vcal}\snapshots - \snapshots + \Vnonlin h(f_{\Vlin}(\snapshots))\|_F^2 + \gamma \|\Vnonlin\|_F^2 \,,
\end{equation}
where $\|\cdot\|_F$ denotes the Frobenius norm. We overload the notation of $h$ in \eqref{eq:Prelim:LinearLSQApproach} to allow $h$ to be evaluated column-wise on the matrix $f_{\Vlin}(\snapshots) = \Vlin^{\top}\snapshots$ to obtain $h(f_{\Vlin}(\snapshots)) \in \mathbb{R}^{\nredmod \times \nsnapshots}$. The regularization term is controlled by $\gamma > 0$ and can prevent overfitting of the weight matrix $\Vnonlin$ to data in $\snapshots$.

\subsection{Greedy construction of quadratic manifolds from data}\label{sec:Prelim:Greedy}
Selecting the $\nred$ leading left-singular vectors $\USVDi{1}, \dots, \USVDi{\nred}$ as columns of $\Vlin$ that span $\Vcal$ can lead to poor results when constructing quadratic manifolds with the optimization problem \eqref{eq:Prelim:LinearLSQApproach}:  Notice that the nonlinear correction term in the decoder \eqref{eq:Prelim:DecoderNonlinear} acts on the projection of the data point onto the first leading left-singular vectors that can miss information that are necessary for the quadratic term to be effective, which is discussed in detail in \cite{SchwerdtnerP2024Greedy}. This is why the work \cite{SchwerdtnerP2024Greedy} proposes a greedy method for selecting $\nred$ left-singular vectors $\USVDi{j_1}, \dots, \USVDi{j_{\nred}}$ with indices $j_1, j_2, \dots, j_{\nred} \in \{1, \dots, \nconsider\}$ from the first $\nconsider \gg \nred$ left-singular vectors $\USVDi{1}, \dots, \USVDi{\nconsider}$. Importantly, the selected left-singular vectors $\USVDi{j_1}, \dots, \USVDi{j_{\nred}}$ with indices $j_1, \dots, j_{\nred}$ are not necessarily the leading $\nred$ left-singular vectors with indices $1, \dots, \nred$.

At each greedy iteration $i = 1, \dots, \nred$, the index $j_i$ is determined as a minimizer of
\begin{equation}\label{eq:Greedy:GreedyOptiProblem}
    \min_{j_i = 1, \dots, m} \min_{\Vnonlin \in \mathbb{R}^{\nfull \times \nredmod}} J(\USVDi{j_i}, \Vlin_{i-1}, \Vnonlin)\,.
    \end{equation}
The matrix $\Vlin_{i-1}=[\USVDi{j_1},\dots,\USVDi{j_{i-1}}]$ contains as columns the left-singular vectors $\USVDi{j_1},\dots,\USVDi{j_{i-1}}$ with indices $j_1, \dots, j_{i - 1}$ that were selected in the previous greedy iterations $1, \dots, i - 1$.
The objective function is
\begin{equation}\label{eq:Greedy:J}
J(\bfv, \Vlin, \Vnonlin) = \|\bfP_{\Vcal \oplus \operatorname{span}\{\bfv\}}\snapshots - \snapshots + \Vnonlin h(f_{[\Vlin, \bfv]}(\snapshots))\|_F^2 + \gamma \|\Vnonlin\|_F^2\,,
\end{equation}
where $\bfP_{\Vcal \oplus \operatorname{span}\{\bfv\}}$ denotes the orthogonal projection operator onto the subspace $\Vcal \oplus \operatorname{span}(\bfv)$ spanned by the columns of $\Vlin$ and the vector $\bfv$. After $\nred$ iterations, the indices $j_1, \dots, j_{\nred}$ are selected and the matrix $\Vlin = [\USVDi{j_1}, \dots, \USVDi{j_{\nred}}]$ is assembled. The weight matrix $\Vnonlin$ is then fitted via~\eqref{eq:Prelim:LinearLSQApproach}.

\subsection{Problem formulation}\label{sec:Prelim:ProbForm}
Current methods \cite{GeelenWW2023Operator,BarnettF2022Quadratic,SchwerdtnerP2024Greedy} for constructing quadratic manifolds operate under the assumption that the data points $\fullstatei{1}, \dots, \fullstatei{\nsnapshots}$ are available all at once at the beginning of the computations, which is typically referred to as batch setting. Furthermore, the algorithms are allowed to make multiple passes over the data samples $\fullstatei{1}, \dots, \fullstatei{\nsnapshots}$. %
In contrast, we are interested in online learning of quadratic manifolds, which means that data points are streamed and have to be processed in one pass as they are received. 

Formally, our setup is that at each iteration $k = 0, 1, 2, 3, \dots$, we have access to a chunk $\chunk{k} = [\fullstatei{k\chunkdim+1}, \dots, \fullstatei{(k+1)\chunkdim}]$ of $\chunkdim \ll \nsnapshots$ data samples.  %
In particular, we have access to $\chunk{k}$ only at iteration $k$ and thus the chunk $\chunk{k}$ has to be processed immediately because it cannot be stored and processed at later iterations.

\section{Online learning of quadratic manifolds}\label{sec:Greedy}
We introduce an online greedy approach that constructs quadratic manifolds from streaming data. The key insight is that all steps of the batch greedy method presented in \cite{SchwerdtnerP2024Greedy} can be reformulated to directly operate on the factors of the SVD of the data matrix rather than on the data points themselves. Thus, it is sufficient to  have access to the SVD of the streamed data samples, instead of requiring direct access to the data samples. Building on this insight, we incrementally update a truncated SVD of the data matrix, to which we then apply the reformulated greedy method  to construct a quadratic manifold.

\subsection{Greedy construction from the SVD of the data matrix}
Recall that $\snapshots = \USVD \SigmaSVD \VSVD^\top$ is the SVD of the data matrix $\snapshots$. 
We now show that the greedy iterations described in Section~\ref{sec:Prelim:Greedy} can be performed with the SVD of $\snapshots$ given by the matrices $\USVD, \SigmaSVD, \VSVD$, without having to assemble the data matrix $\snapshots$.

At greedy iteration $i = 1, \dots, \nred$, we define the index sets $\indexin=\{j_1, \dots, j_i\}$ and $\indexout=\{1, \dots, \min(\nfull, \nsnapshots)\} \setminus\indexin$, where $j_1, \dots, j_i$ are the indices of the left-singular vectors that have been selected by the greedy algorithm up to iteration $i$. 
Based on the index sets $\indexin$ and $\indexout$, we permute the SVD as 
    \begin{align}
      \label{eq:OurMethod:PermSVD}
      \snapshots = \USVD \SigmaSVD \VSVD^\top =
      \begin{bmatrix}
        \USVD_{\indexin} & \USVD_{\indexout}
      \end{bmatrix}
      \begin{bmatrix}
        \SigmaSVD_{\indexin} & 0 \\
        0 & \SigmaSVD_{\indexout}
      \end{bmatrix}
      \begin{bmatrix}
        \VSVD_{\indexin} & \VSVD_{\indexout}
      \end{bmatrix}^\top,
    \end{align}
where $\USVD_{\indexin}=[\USVDi{j_1}, \dots, \USVDi{j_{i}}]$, $\SigmaSVD_{\indexin}$ contains $\SigmaSVDi{j_1}, \dots, \SigmaSVDi{j_{i}}$ on its diagonal, and $\VSVD_{\indexin} = [\VSVDi{j_1}, \dots, \VSVDi{j_i}]$. Analogously, the matrices $\USVD_{\indexout}, \VSVD_{\indexout}$, and $\SigmaSVD_{\indexout}$ contain the  singular vectors and singular values with indices in $\indexout$.
Let now $\Vlin_{\indexin} = [\USVDi{j_1}, \dots, \USVDi{j_i}]$ be a basis matrix that spans the subspace $\Vcal_{\indexin} \subset \R^{\nfull}$.
We can then write the orthogonal projection  
 of $\snapshots$ onto $\Vcal_{\indexin}$ as
\[
\bfP_{\Vcal_{\indexin}}\snapshots = \USVD_{\indexin}\SigmaSVD_{\indexin}\VSVD_{\indexin}^\top\,.
\]
Furthermore, because the left-singular vectors are orthonormal, we obtain that the projection error of $\snapshots$ onto $\Vcal_{\indexin}$ is
    \begin{align}
      \bfP_{\Vcal_{\indexin}}\snapshots - \snapshots &= \USVD_{\indexout} \SigmaSVD_{\indexout} \VSVD_{\indexout}^\top\,.
    \end{align}
Finally, we can represent the reduced data points of $\snapshots$ given by the encoder $f_{\Vcal_{\indexin}}$ as $f_{\Vcal_{\indexin}}(\snapshots) = \SigmaSVD_{\indexin}\VSVD_{\indexin}$.

Equipped with these reformulations, we can state the the least-squares problem~\eqref{eq:Prelim:LinearLSQApproach} using only factors of the SVD of $\snapshots$ as 
    \begin{align}
        \label{eq:OurMethod:LstsqUsingSVD}
        \min_{\Vnonlin \in \mathbb{R}^{\nfull \times \nredmod}} \|\USVD_{\indexout} \SigmaSVD_{\indexout} \VSVD_{\indexout}^\top + \Vnonlin \featuremap(\SigmaSVD_{\indexin}\VSVD^\top_{\indexin})\|_F^2 + \gamma \|\Vnonlin\|_F^2 \,.
    \end{align}
Similarly, because the greedy iterations evaluate the objective function~\eqref{eq:Greedy:J} only at left-singular vectors, we can represent \eqref{eq:Greedy:J} as 
\begin{align}
\label{eq:OurMethod:ObjFunUsingSVD}
      J'(j, \Vlin_{\indexin}, \Vnonlin)=\|\USVD_{\indexout\setminus \{j\}} \SigmaSVD_{\indexout\setminus \{j\}} \VSVD_{\indexout\setminus \{j\}}^\top + \Vnonlin \featuremap(\SigmaSVD_{\indexin \cup \{j\}}\VSVD^\top_{\indexin \cup \{j\}})\|_F^2 + \gamma \|\Vnonlin\|_F^2 \,,
  \end{align}
where the first argument is now the index $j$ of the left-singular vector $\USVDi{j}$ instead of a generic vector $\bfv$ as in \eqref{eq:Greedy:J}. Notice that $\indexout\setminus \{j\}$ is the set $\indexout$ with the element $j$ removed. 
Because each greedy iteration just minimizes \eqref{eq:Greedy:J} over the left-singular vectors $\USVDi{1}, \dots, \USVDi{\nconsider}$ according to the optimization problem given in \eqref{eq:Greedy:GreedyOptiProblem}, the objective $J$ in \eqref{eq:Greedy:J} can be replaced with  the objective $J^{\prime}$ given in \eqref{eq:OurMethod:ObjFunUsingSVD} that uses the factors $\USVD, \SigmaSVD, \VSVD$ of the SVD of $\snapshots$ only.

\subsection{Greedy construction from a truncated SVD of the data matrix}
Let us now consider a truncated SVD of the data matrix. %
Recall the full SVD in \eqref{eq:OurMethod:PermSVD} that is permuted based on the index sets $\indexin$ and $\indexout$. 
We introduce a third index set that we denote as $\indextrunc$ and that we define as $\indextrunc=\{1, \dots, \nconsider\} \setminus \indexin$. The three index sets $\indexin, \indextrunc$, and $\indexout\setminus\indextrunc$ lead to the SVD of $\snapshots$ as
\begin{align}
      \label{eq:OurMethod:PermPermSVD}
      \snapshots = \USVD \SigmaSVD \VSVD^\top =
      \begin{bmatrix}
        \USVD_{\indexin} & \USVD_{\indextrunc} & \USVD_{\indexout\setminus\indextrunc}
      \end{bmatrix}
      \begin{bmatrix}
        \SigmaSVD_{\indexin} & 0 & 0\\
        0 & \SigmaSVD_{\indextrunc} & 0\\
        0 & 0 & \SigmaSVD_{\indexout\setminus\indextrunc}
      \end{bmatrix}
      \begin{bmatrix}
        \VSVD_{\indexin} & \VSVD_{\indextrunc} & \VSVD_{\indexout\setminus\indextrunc}
      \end{bmatrix}^\top.
    \end{align}
Leaving out the singular vectors and singular values with index in $\indexout\setminus\indextrunc$ leads to a truncated SVD 
\begin{equation}\label{eq:TruncatedSVD}
\USVD_{\nconsider} \SigmaSVD_{\nconsider} \VSVD^\top_{\nconsider} =
      \begin{bmatrix}
        \USVD_{\indexin} & \USVD_{\indextrunc} 
      \end{bmatrix}
      \begin{bmatrix}
        \SigmaSVD_{\indexin} & 0 \\
        0 & \SigmaSVD_{\indextrunc} 
      \end{bmatrix}
      \begin{bmatrix}
        \VSVD_{\indexin} & \VSVD_{\indextrunc} \end{bmatrix}^\top,
\end{equation}
which includes the singular vectors and singular values with index $1, \dots, \nconsider$ only that are in $\indexin \cup \indextrunc$. The truncated SVD \eqref{eq:TruncatedSVD} has at most rank $\nconsider$. We sometimes refer to $\nconsider$ as the truncation dimension.

Using only the truncated SVD in the least-squares problem  \eqref{eq:OurMethod:LstsqUsingSVD} based on the set $\indextrunc$ instead of $\indexout$ leads to 
\begin{align}
\label{eq:OurMethod:LstsqUsingTruncSVD}
    \min_{\Vnonlin \in \mathbb{R}^{\nfull \times \nredmod}} \|  \USVD_{\indextrunc} \SigmaSVD_{\indextrunc} \VSVD_{\indextrunc}^\top+ \Vnonlin \featuremap(\SigmaSVD_{\indexin}\VSVD^\top_{\indexin})\|_F^2 + \gamma \|\Vnonlin\|_F^2 \,.
\end{align}
By leaving out the difference
\begin{equation}\label{eq:OGreedyDiff}
\USVD_{\indexout\setminus\indextrunc} \SigmaSVD_{\indexout\setminus\indextrunc} \VSVD_{\indexout\setminus\indextrunc}^\top = \USVD_{\indexout} \SigmaSVD_{\indexout} \VSVD_{\indexout}^\top - \USVD_{\indextrunc} \SigmaSVD_{\indextrunc} \VSVD_{\indextrunc}^\top\,,
\end{equation}
from the objective of \eqref{eq:OurMethod:LstsqUsingSVD}, it is sufficient to have available the truncated SVD \eqref{eq:TruncatedSVD}
of $\snapshots$. 
Similarly, we can approximate the objective $J^{\prime}$ given in \eqref{eq:OurMethod:ObjFunUsingSVD} as
\begin{equation}\label{eq:OurMethod:ObjFunUsingTruncSVD}
\hat{J}'(j, \Vlin_{\indexin}, \Vnonlin)=\left\|\USVD_{\indextrunc\setminus \{j\}} \SigmaSVD_{\indextrunc\setminus \{j\}} \VSVD_{\indextrunc\setminus \{j\}}^\top + \Vnonlin \featuremap(\SigmaSVD_{\indexin \cup \{j\}}\VSVD^\top_{\indexin \cup \{j\}})\right\|_F^2 + \gamma \|\Vnonlin\|_F^2 \,,
\end{equation}
\newcommand{\featSin}{\featuremap(\SigmaSVD_{\indexin}\VSVD_{\indexin}^\top)}
which relies on the truncated SVD \eqref{eq:TruncatedSVD} only. 

In summary, by using the objective \eqref{eq:OurMethod:ObjFunUsingTruncSVD} for selecting the $\Vlin$ via the greedy construction and the minimization problem \eqref{eq:OurMethod:LstsqUsingTruncSVD} to find $\Vnonlin$, we rely on the truncated SVD \eqref{eq:TruncatedSVD} of the data matrix $\snapshots$ only. 
It is important to note that when solving \eqref{eq:OurMethod:LstsqUsingTruncSVD} and evaluating \eqref{eq:OurMethod:ObjFunUsingTruncSVD}, the low-rank matrices $\USVD_{\indextrunc} \SigmaSVD_{\indextrunc} \VSVD_{\indextrunc}^\top$ and $\USVD_{\indextrunc\setminus \{j\}} \SigmaSVD_{\indextrunc\setminus \{j\}} \VSVD_{\indextrunc\setminus \{j\}}^\top$, respectively, are not multiplied to avoid having to assemble large matrices and instead the corresponding least-squares problems are solved using the factorized forms. Moreover, in~\cite{SchwerdtnerP2024Greedy}, it is shown that $\hat{J}'$ can be evaluated without computing $\Vnonlin$ leading to lower computational costs.

The following proposition bounds the difference between the minimizer of \eqref{eq:OurMethod:LstsqUsingSVD} and the minimizer when the truncated SVD is used and thus \eqref{eq:OurMethod:LstsqUsingTruncSVD} is solved instead.

\begin{proposition} %
The Frobenius norm of the difference between the minimizers of~\eqref{eq:OurMethod:LstsqUsingSVD} and~\eqref{eq:OurMethod:LstsqUsingTruncSVD} is bounded by $\alpha\big\|{\SigmaSVD^{\indexout\setminus\indextrunc}}\big\|_F$, where
\begin{align*}
\alpha=\frac{\sigma_{\max}(\featuremap(\SigmaSVD_{\indexin}\VSVD^\top_{\indexin})))}{\sigma_{\min}\left(\featuremap(\SigmaSVD_{\indexin}\VSVD^\top_{\indexin}) \featuremap(\SigmaSVD_{\indexin}\VSVD_{\indexin}^\top)^\top + \gamma I\right)}.
\end{align*}\label{prop:BoundProp}
\end{proposition}
\begin{proof} %
    The minimizers $\Vnonlin^{(3.3)}$ and $\Vnonlin^{(3.7)}$ of \eqref{eq:OurMethod:LstsqUsingSVD} and \eqref{eq:OurMethod:LstsqUsingTruncSVD}, respectively are unique for $\gamma > 0$ and given by
    \begin{align*}
    \Vnonlin^{(3.3)} &=\left(\featSin \featSin^\top+\gamma I\right)^{-1} \featSin \VSVD_{\indexout} \SigmaSVD_{\indexout} \USVD_{\indexout}^\top, \\
    \Vnonlin^{(3.7)} &=\left(\featSin \featSin^\top+\gamma I\right)^{-1} \featSin \VSVD_{\indextrunc} \SigmaSVD_{\indextrunc} \USVD_{\indextrunc}^\top.
    \end{align*}
    Therefore, the difference $\mathbf{\Delta}_{\Vnonlin}=\Vnonlin^{(3.3)}-\Vnonlin^{(3.7)}$ is given by
    \begin{align*}
    \mathbf{\Delta}_{\Vnonlin}&=\left(\featSin \featSin^\top+\gamma I\right)^{-1} \featSin\left(\VSVD_{\indexout} \SigmaSVD_{\indexout} \USVD_{\indexout}^\top-\VSVD_{\indextrunc} \SigmaSVD_{\indextrunc} \USVD_{\indextrunc}^\top \right)\\
    &=\left(\featSin \featSin^\top+\gamma I\right)^{-1} \featSin \USVD_{\indexout\setminus\indextrunc} \SigmaSVD_{\indexout\setminus\indextrunc} \VSVD_{\indexout\setminus\indextrunc}^\top,
    \end{align*}
    which has norm
    \begin{equation}\label{eq:BoundTrunc}\begin{aligned}
    {\|\mathbf{\Delta}_{\Vnonlin}\|}_F&={\left\|\left(\featSin \featSin^\top+\gamma I\right)^{-1} \featSin \USVD_{\indexout\setminus\indextrunc} \SigmaSVD_{\indexout\setminus\indextrunc} \VSVD_{\indexout\setminus\indextrunc}^\top\right\|}_F \\
    &\le
    \frac{\sigma_{\max}(h(\SigmaSVD_{\indexin}\VSVD_{\indexin})))}{\sigma_{\min}\left(h(\SigmaSVD_{\indexin}\VSVD_{\indexin}) h(\SigmaSVD_{\indexin}\VSVD_{\indexin})^\top + \gamma I\right)}
    \big\|{\SigmaSVD^{\indexout\setminus\indextrunc}}\big\|_F\,.
    \end{aligned}
    \end{equation}
\end{proof}
 
The bound \eqref{eq:BoundTrunc} given in Proposition~\ref{prop:BoundProp} scales with the sum of the discarded singular values $\SigmaSVD^{\indexout\setminus\indextrunc}$, which are all smaller than any of the singular values that are retained after truncation.   

\subsection{Incremental updates to truncated SVD of the data matrix}\label{sec:Streaming:OnlineTruncSVD}
We now discuss an incremental updated version of the truncated SVD of the data matrix as data chunks $\chunk{0}, \chunk{1}, \chunk{2}, \dots$ are received over the iterations $k = 0, 1, 2, 3, \dots$; see problem formulation in Section~\ref{sec:Prelim:ProbForm}. Each chunk $\chunk{k}$ is of size $\nfull \times \chunkdim$. At iteration $k = 0$, we have available an initial truncated SVD \eqref{eq:TruncatedSVD}, which we denote with $\USVD_{\nconsider}^{(-1)} \in \mathbb{R}^{\nfull \times \nconsider}, \SigmaSVD_{\nconsider}^{(-1)} \in \mathbb{R}^{\nconsider \times \nconsider}, \VSVD_{\nconsider}^{(-1)} \in \mathbb{R}^{0 \times \nconsider}$. %

At iteration $k = 0, 1, 2, 3, \dots$, we update the matrices $\USVD_{\nconsider}^{(k-1)},\SigmaSVD_{\nconsider}^{(k-1)}, \VSVD_{\nconsider}^{(k-1)}$ with a new data chunk $\chunk{k}$ to obtain $\USVD_{\nconsider}^{(k)},\SigmaSVD_{\nconsider}^{(k)}, \VSVD_{\nconsider}^{(k)}$. 
To perform the update, we follow the approach introduced in~\cite{BakerGV2012Low-rank}. 
The first step is to process $\chunk{k}$ by computing the QR decomposition  
\begin{equation}\label{eq:Streaming:QR}
  \bfQ^{(k)}\bfR^{(k)} = [\USVD_{\nconsider}^{(k-1)}\SigmaSVD_{\nconsider}^{(k-1)}, \chunk{k}] \in \mathbb{R}^{\nfull \times (\nconsider+\chunkdim)}\,,
  \end{equation}
  with the factors $\bfQ^{(k)}$ and $\bfR^{(k)}$.  Furthermore, the matrix $\VSVD_{\nconsider}^{(k-1)}$ of right-singular vectors is expanded as
\begin{equation}\label{eq:Streaming:VSVDExpansion}
  \hat{\VSVD}_{\nconsider}^{(k-1)} = \begin{bmatrix}
        \VSVD_{\nconsider}^{(k-1)} & 0 \\
        0 & \bfI_{\chunkdim} \\
      \end{bmatrix}\,,
      \end{equation}
  where $\bfI_{\chunkdim}$ is the $\chunkdim \times \chunkdim$ identity matrix. 
The second step takes the SVD of $\bfR^{(k)}$ to obtain
\begin{equation}\label{eq:Streaming:RSVD}
  \bfR^{(k)} = \GU \GSigma{\GV}^\top\,.
  \end{equation}
The matrix of left- and right-singular vectors of the SVD \eqref{eq:Streaming:RSVD} act as transformations to compute
\begin{equation}\label{eq:Streaming:UpdateStepUV}
 \USVD_{\nconsider}^{(k)} = [\bfQ^{(k)}\GU]_{1:\nconsider}\,,\qquad
 \VSVD_{\nconsider}^{(k)} = [\hat \VSVD_{\nconsider}^{(k-1)}\GV]_{1:\nconsider}\,,
 \end{equation}
 where $[\,]_{1:\nconsider}$ means that the first $\nconsider$ columns are used only. 
 The diagonal matrix of the updated singular values $\SigmaSVD_{\nconsider}^{(k)}$ is set to the diagonal matrix that has as diagonal elements the first $\nconsider$ diagonal elements of $\GSigma$.

 The incremental updating of the truncated SVD introduces errors. Roughly speaking, the error incurred by the incremental updating of the SVD is proportional to the smallest positive singular value of the truncated SVD. Thus, choosing $\nconsider$ large enough can keep the error of the incremental update under control; see  \cite{ChahlaouiGV-D2003Recursive, Baker2004Block,BakerGV2012Low-rank} for a priori and a posteriori error bounds. 

We now discuss the computational costs and storage requirements of the steps for updating the matrices $\USVD_{\nconsider}^{(k-1)} \in \mathbb{R}^{\nfull \times \nconsider},\SigmaSVD_{\nconsider}^{(k-1)} \in \mathbb{R}^{\nconsider \times \nconsider}, \VSVD_{\nconsider}^{(k-1)} \in \mathbb{R}^{k\chunkdim \times q}$  
to $\USVD_{\nconsider}^{(k)}\in \mathbb{R}^{\nfull \times \nconsider},\SigmaSVD_{\nconsider}^{(k)} \in \mathbb{R}^{\nconsider \times \nconsider}, \VSVD_{\nconsider}^{(k)} \in \mathbb{R}^{(k + 1)\chunkdim \times q}$ for a chunk $\chunk{k}$. The costs of computing the QR decomposition  \eqref{eq:Streaming:QR} scale as $\mathcal{O}(\nfull (\nconsider + \chunkdim)^2)$. The SVD of $\bfR^{(k)}$ computed in \eqref{eq:Streaming:RSVD} has costs that scale as $\mathcal{O}((\nconsider+\chunkdim)^3)$.
At iteration $k$, the costs of the matrix multiplications in \eqref{eq:Streaming:UpdateStepUV} %
scale as $\mathcal{O}(\nfull\nconsider^2)$ and $\mathcal{O}(k\chunkdim\nconsider^2)$, respectively. %
Thus, at iteration $k$, the costs scale linearly in the dimension $\nfull$ of the data points and linearly in $k\chunkdim$ but we avoid costs that scale as $\mathcal{O}(\max(\nfull,k\chunkdim) \times \min(\nfull, k\chunkdim)^2)$. The storage requirement is dominated by the matrices $\USVD_{\nconsider}^{(k)}$ and $\VSVD_{\nconsider}^{(k)}$, since the full matrix $\snapshots$ is never used and the chunks $\chunk{k}$ are not accumulated in memory. Therefore, at iteration $k$, the storage requirement only scales linearly in $\nfull$ and the number of data samples $(k+1)\chunkdim$ received up to iteration $k$, that is our storage requirement scales as $\mathcal{O}(k \chunkdim \maxrank + \nfull \maxrank)$. This is in stark contrast to  storing the data matrix at iteration $k$, which scales as $\mathcal{O}(\nfull k\chunkdim)$, whereas our approach has storage requirements that avoid the multiplicative scaling of dimension $\nfull$ and number of data points $(k+1) \chunkdim$ received up to iteration $k$  and instead achieves an additive scaling as $((k+1) \chunkdim + \nfull)  \maxrank$.

\subsection{Online greedy construction of quadratic manifolds from streaming data}
Building on incremental updates to a truncated SVD of the data matrix allow us to derive Algorithm~\ref{alg:OurMethod:streaming_qm} for constructing quadratic manifolds from streaming data. The algorithm takes as input a data chunk $\chunk{k}$, the current SVD $\USVD^{(k-1)}_\nconsider, \SigmaSVD^{(k-1)}_\nconsider, \VSVD^{(k-1)}_\nconsider$, the truncation dimension $\maxrank$, the latent dimension $\nred$, as well as the regularization parameter $\gamma$, and the feature map $\featuremap$.
The algorithm processes the incoming data chunk as follows. As the new data chunk $\chunk{k}$ is received, first the SVD is updated (lines 4-7) and after that the greedy construction of a quadratic manifold using the updated SVD is performed. Note that the quadratic manifold does not have to be updated at each step. In our numerical experiments, we only update the SVD as we process the data chunks and compute the quadratic manifold only once after all chunks have been processed.

\begin{algorithm}[t]
  \caption{Online greedy construction of quadratic manifolds from streaming data}
  \label{alg:OurMethod:streaming_qm}
  \begin{algorithmic}[1]
    \Procedure{StreamingQM}{$\chunk{k}, \USVD^{(k-1)}_\nconsider, \SigmaSVD^{(k-1)}_\nconsider, \VSVD^{(k-1)}_\nconsider, \maxrank, \nred, \gamma, \featuremap$}
    \State{\textit{// Update SVD}}
    \State{$\bfQ^{(k)}, \bfR^{(k)}, \hat{\VSVD}_{\nconsider}^{(k-1)}$ using \eqref{eq:Streaming:QR} and \eqref{eq:Streaming:VSVDExpansion}.}
    \State{Compute $\GU$ and $\GV$ as left and right singular vectors of $\bfR^{(k)}$.}
    \State{Set $\USVD_\nconsider^{(k)}=\bfQ^{(k)} \GU$, $\VSVD_\nconsider^{(k)}=\hat \VSVD_\nconsider^{(k-1)} \GV$, and $\SigmaSVD_\nconsider^{(k)}={\GU}^{\top} \bfR^{(k)}\GV$.}
    \State{Truncate $\USVD_\nconsider^{(k)}$, $\VSVD_\nconsider^{(k)}$, and $\SigmaSVD_\nconsider^{(k)}$ to keep a rank $\maxrank$ approximation.}
    \State{\textit{// Optional: update quadratic manifold}}
    \State{Set $\mathcal{I}_0=\{\}, \mathcal{\breve I}_0=\{1, \dots, \maxrank\}, \Vlin_0 = []$}
    \For{$i=1,\dots, \nred$}
    \State{Compute $\USVDi{j_i}^{(k)}$ that minimizes~\eqref{eq:OurMethod:ObjFunUsingTruncSVD} over all $\USVDi{j^{1}}^{(k)}, \dots, \USVDi{j^\maxrank}^{(k)}$ and $\Vnonlin \in \mathbb{R}^{\nfull \times \nredmod}$}
    \State{Set $\mathcal{I}_i=\{j_1, \dots, j_i\}$ and $\mathcal{\breve I}_i=\{1, \dots, \nsnapshots\}\setminus \mathcal{I}_i$}
    \State{Set $\Vlin_i = [\USVDi{1}^{(k)}, \dots, \USVDi{j_k}^{(k)}]$}
    \EndFor
    \State{Set $\Vlin^{(k)}=[\USVDi{j_1}^{(k)}, \dots, \USVDi{j_r}^{(k)}]$} %
    \State{Compute $\Vnonlin^{(k)}$ via the regularized least-squares problem \eqref{eq:OurMethod:LstsqUsingTruncSVD}.}
    \EndProcedure
  \end{algorithmic}
\end{algorithm}

In Algorithm~\ref{alg:OurMethod:streaming_qm}, we optimize for $\USVDi{j_i}^{(k)}$ that minimizes~\eqref{eq:OurMethod:ObjFunUsingTruncSVD} over all $\USVDi{j^{1}}^{(k)}, \dots, \USVDi{j^\maxrank}^{(k)}$, that is, over all computed left singular vectors. In~\cite{SchwerdtnerP2024Greedy}, we limit the search to only a small subset of the singular vectors, to trade off runtime and accuracy. Since in the streaming variant, the set of computed singular vectors  $\nconsider$ is already much smaller than $\nfull$ and $\nsnapshots$, we search all over all $\nconsider$ computed left singular vectors.

\section{Numerical experiments}
\label{sec:numerical_experiments}
We demonstrate the online greedy method with three numerical experiments, ranging from Hamiltonian wave problems to the Kelvin-Helmholtz instability. In the last example, we train a quadratic manifold on more than one Petabyte of data. We run the numerical experiments on Xeon Sapphire Rapids 52-core processors with 256GB of main memory. The simulation and data-processing of our last example is carried out on an NVIDIA
H-100 GPU with 80GB of memory. A python implementation using jax~\cite{BradburyFHJLMNPVW2018JAX} is available at \url{github.com/Algopaul/sqm_demo_gpu}.  %

\subsection{Hamiltonian interacting pulse}
\label{sec:NumExp:Hamiltonian}
\begin{figure}
\begin{tabular}{cc}
\resizebox{0.40\textwidth}{!}{\input{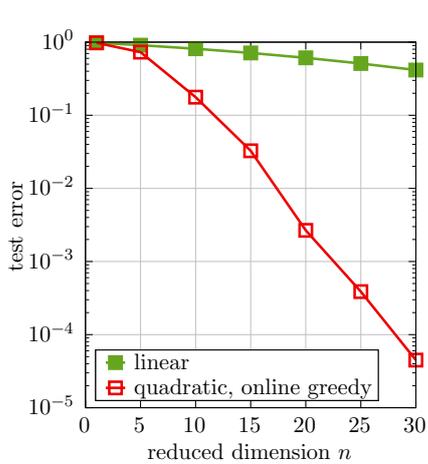}}
& ~~~~~\resizebox{0.48\textwidth}{!}{\input{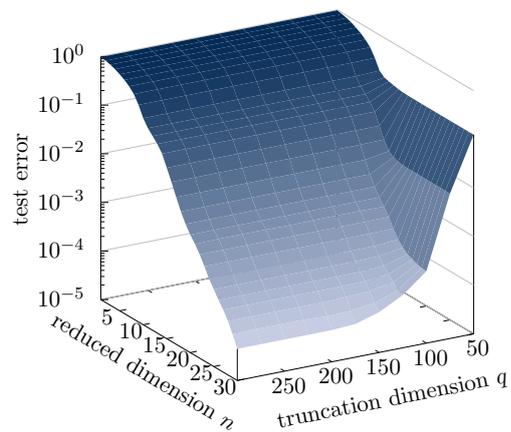}}\\
\scriptsize (a) test error & \scriptsize (b) reduced versus truncation dimension
\end{tabular}
\caption{Hamiltonian wave: (a) The proposed online greedy method constructs a quadratic manifold from streaming data that achieves orders of magnitude low approximation errors \eqref{eq:NumExp:RelError} on test data than  linear dimensionality reduction with the SVD. (b) As the reduced dimension $\nred$ of the quadratic manifold is increased, the truncation dimension $\nconsider$ has to be increased too so that the online greedy method can utilize higher-order left-singular vectors of the data matrix.}
\label{fig:NumExp:Hamiltonian:Error}
\end{figure}

\begin{figure}
    \centering
    \resizebox{1.0\textwidth}{!}{\input{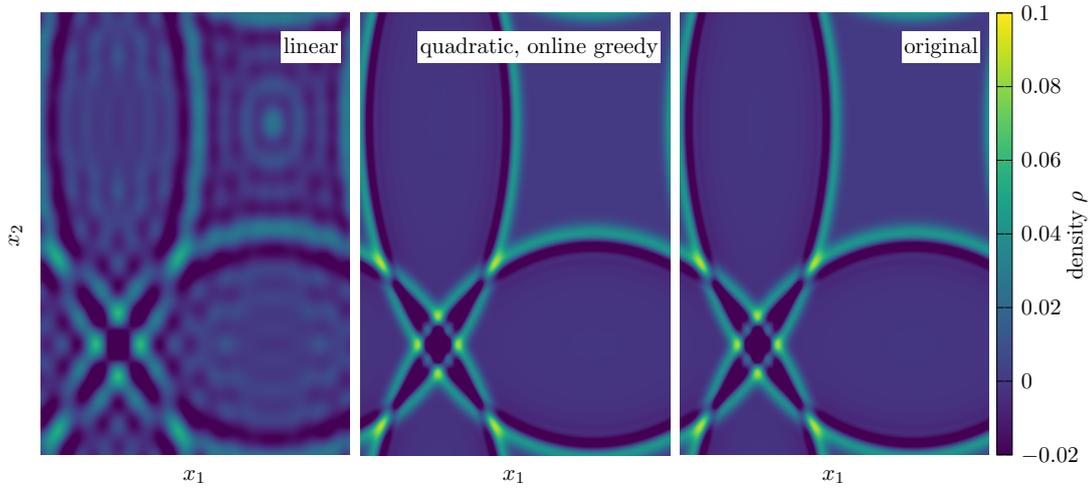}}
    \caption{Hamiltonian wave: The online greedy method constructs quadratic manifolds from streaming data that achieve visibly better approximations of unseen test data than linear dimensionality reduction. Approximations shown for reduced dimension $\nred=20$ and time $t=6$.}
    \label{fig:hamilflow:reconstruction}
\end{figure}

\begin{figure}
    \centering
    \resizebox{1.0\textwidth}{!}{\input{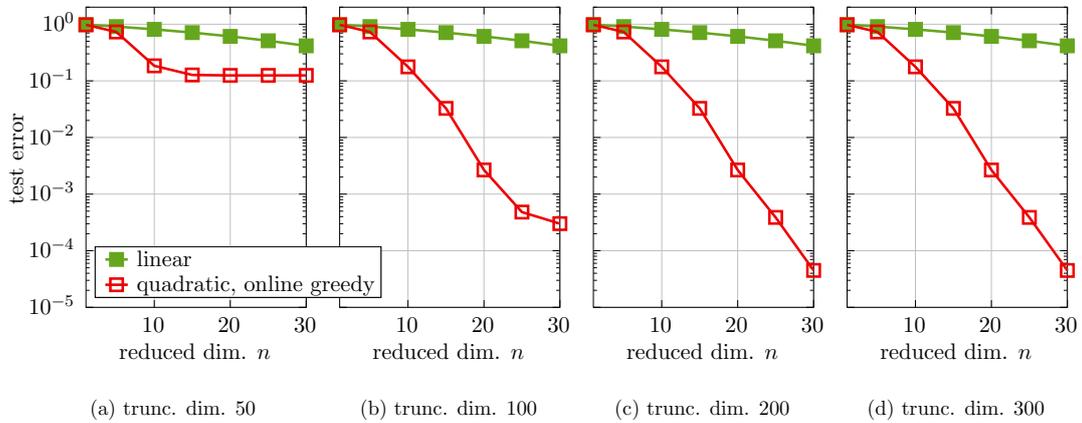}}
    \caption{Hamiltonian wave: The plots show that the truncation dimension $\nconsider$ and the reduced dimension $\nred$ have to increase in tandem so that the online greedy method can leverage higher-order left-singular vector for larger $\nred$.}
    \label{fig:hamilflow:test_errs}
\end{figure}
We consider data samples that are numerical solutions of the acoustic wave equation in Hamiltonian form,
  \begin{align}
\label{eq:hamiltonian_wave_pde}
      \partial_t \rho(t, \bfx) &= -\nabla v(t, x), \\
      \partial_t v(t, \bfx) &= -\nabla \rho(t, \bfx),\\
      \rho(0, \bfx) &= \rho_0(\bfx; \mu), \\
      v(0, \bfx) &= 0,
  \end{align}
where $\rho: [0, T] \times \Omega \to \mathbb{R}$ denotes the density and $v: [0, T] \times \Omega \to \mathbb{R}^2$ the velocity field.
The time interval is $[0, T]$ with $T = 8$ and the spatial domain is $\Omega = [-4, 4)^2$. We impose periodic boundary conditions. The initial condition $\rho_0$ depends on a parameter $\mu \in [0, 1]$,
\begin{align}
    \rho_0(\bfx, \mu) = \exp\left(-(\mu+6)^2 \left((x_1-2)^2+(x_2-2)^2\right)\right),
  \end{align}
  where $\bfx = [x_1, x_2]^T \in \Omega$. 
We employ a central second order finite difference scheme with 600 grid points in each spatial direction, resulting in a state-space dimension of $\nfull = 1,080,000$. Time is discretized with a fourth-order Runge-Kutta method with time-step size $\delta t = 5 \times 10^{-3}$. A data point $\fullstate_i$ corresponds to the solution fields $\rho$ and $v$ evaluated at the $\nfull$ grid points at one of the $T / \delta t = 1601$ time points that are equidistantly distributed in $[0, T]$.

We generate 101 trajectories corresponding to the initial conditions with parameters $\mu \in \{0, 1/100, 2/100, \dots, 1\}$.
We divide the 101 trajectories into 99 training trajectories that lead to the data points in $\snapshots$, one validation trajectory corresponding to $\mu = 1/4$ that we collect in $\snapshots^{\text{valid}}$ and one test trajectory corresponding to $\mu = 3/4$ that we store in $\snapshots^{\text{test}}$.
The size of the training data $\snapshots$ is approximately 1.37\,TB. 

We process the training data with the online greedy method over the iterations $k = 0, 1, 2, 3, \dots$,  where we receive chunks $\chunk{0}, \chunk{1}, \chunk{2}, \chunk{3}, \dots $ with size $\chunkdim = 347$. %
We learn quadratic manifolds of dimensions $\nred \in \{1, 5, 10, 15, \dots, 30\}$. %
For now, we set the truncation dimension to $\nconsider = 300$ because the largest considered reduced dimension $\nred = 30$ is multiplied by a factor 10, which is the convention also used in the work \cite{SchwerdtnerP2024Greedy}; we will study the effect of $\nconsider$ in more detail below.
The test error that we report is the relative error of approximating test trajectories
\begin{equation}\label{eq:NumExp:RelError}
  e(\snapshots^{\text{test}}) = \frac{\|\hat{\snapshots} - \snapshots^{\text{test}}\|_F^2}{\|\snapshots^{\text{test}}\|_F^2}\,,
  \end{equation}
  where $\snapshots^{\text{test}} = [\fullstatei{1}^{\text{test}}, \fullstatei{2}^{\text{test}}, \dots]$ contains the test trajectory and $\hat{\snapshots} = g(f(\snapshots^{\text{test}}))$ is the test trajectory approximated with the encoder function $f$ and the decoder function $g$.  
The validation trajectory $\snapshots^{\text{valid}}$ is used to select the regularization parameter $\gamma$ over a sweep $10^{-8}, 10^{-6}, \dots,10^{0}$ that leads to the lowest relative error that is analogously defined to \eqref{eq:NumExp:RelError} but for the validation data $\snapshots^{\text{val}}$. Based on this procedure, we select $\gamma = 10^{-8}$ in this example. 

Figure~\ref{fig:NumExp:Hamiltonian:Error}a shows the relative error \eqref{eq:NumExp:RelError} of approximating the test trajectory. We compare the quadratic manifold to linear dimensionality reduction that uses the encoder and decoder function given by the SVD truncated at $\nred$ as discussed in Section~\ref{sec:Prelim:EnDeLinear}. Note that the linear encoder and decoder also use the incremental SVD as described in Section~\ref{sec:Streaming:OnlineTruncSVD}. %
The results in Figure~\ref{fig:NumExp:Hamiltonian:Error}a show that the quadratic manifold achieves orders of magnitude lower test errors \eqref{eq:NumExp:RelError} than linear dimensionality reduction. The results are in agreement with the plots shown in Figure~\ref{fig:hamilflow:reconstruction}, where the linear dimensionality reduction leads to visible artifacts, while the quadratic-manifold approximation is visually indistinguishable from the original density field.
  
We now investigate how the truncation dimension $\nconsider$ of the incremental SVD  influences the error of the online greedy method. Figure~\ref{fig:NumExp:Hamiltonian:Error}b shows the test error \eqref{eq:NumExp:RelError} over reduced dimensions $\nred \in \{1, 2, \dots, 30\}$ and truncation dimensions $\nconsider \in \{50, 100, \dots, 300\}$. A large reduced dimension $\nred$ combined with a low truncation dimension $\nconsider$ leads to a large error, which decreases as $\nconsider$ is increased. This indicates that the truncation dimension $\nconsider$ and the reduced dimension $\nred$ have to increase in tandem so that the online greedy method can leverage higher-order left-singular vectors. 
Figure~\ref{fig:hamilflow:test_errs} shows slices of Figure~\ref{fig:NumExp:Hamiltonian:Error}b that further emphasize that $\nconsider$ and $\nred$ depend on each other.
The reason that a low truncation dimension $\nconsider$ leads to larger errors is twofold. First, by truncating too early, the truncated terms  \eqref{eq:OGreedyDiff} still contain information that are important for fitting the quadratic manifold. 
Second, the truncated incremental SVD underlying the online greedy method incurs errors over the online iterations $k = 0, 1, 2, 3, \dots$; see discussion in Section~\ref{sec:Streaming:OnlineTruncSVD}. The errors are typically of the order of the first singular value that is truncated, which means that truncating too early leads to quick error accumulation over the iterations. %
Figure~\ref{fig:hamilflow:singvals} shows the approximations of singular values and their errors obtained with the incremental truncated SVD for $\nconsider \in \{50, 100, 200, 300\}$. 
The absolute errors of the singular values are computed with respect to the  singular value obtained with $\nconsider = 300$. The leading singular values are computed to a high accuracy while the error increases for later singular values. Furthermore,  the results show that the error in the approximations of the singular values depends on the truncation dimension $\nconsider$, which is in agreement with the discussion above.

\begin{figure}
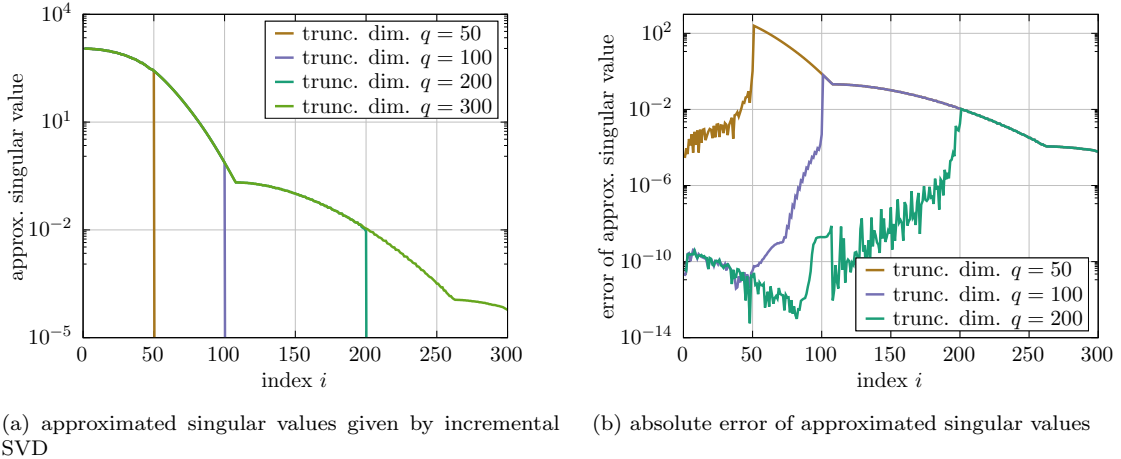

  \centering
  \begin{tabular}{p{0.5\textwidth} p{0.5\textwidth}}
    \resizebox{0.48\textwidth}{!}{\input{PlotSources/hamilflow/hamilflow_singvals.tex}} &
    \resizebox{0.48\textwidth}{!}{\input{PlotSources/hamilflow/hamilflow_singvals_errs.tex}} \\
    \scriptsize (a) approximated singular values given by incremental SVD &
    \scriptsize (b) absolute error of approximated singular values \\
  \end{tabular}
  \caption{Hamiltonian wave: Because the leading singular values are approximated well by the incremental SVD, it is sufficient to truncate the SVD early in the online greedy method. Increasing the truncation dimension $\nconsider$ reduces the error of the approximated singular values.}
  \label{fig:hamilflow:singvals}
\end{figure}

\subsection{Turbulent channel flow at a high Reynolds numbers}
We now consider data samples that represent three-dimensional velocity fields of a turbulent channel flow, which are computed using AMR-Wind \cite{BrazellAVCSEH2021AMR-Wind}.
We use a wall-modeled large eddy simulation at a friction Reynolds number 5200, discretized with a staggered finite volume method into a grid of $384 \times 192 \times 32$ cells. The dimension of the data points is $\nfull = 3 \times 384 \times 192 \times 32 = 7,077,888$.
We have a total of $\nsnapshots = 10000$ data points, which we uniform randomly split into 9000 training data samples in $\snapshots$, 500 validation data samples in $\snapshots^{\text{valid}}$, and 500 test samples in $\snapshots^{\text{test}}$.
The training data $\snapshots$ is of size of  about 510 GB. We receive data samples via chunks of size $\chunkdim = 52$. %
The setup of computing the quadratic manifold is the same as in Section~\ref{sec:NumExp:Hamiltonian}.

Figure~\ref{fig:channelflow:test_errs} shows the relative error \eqref{eq:NumExp:RelError} on the test data for the quadratic manifold with dimension $\nred \in \{1, 5, 10, 15, 20, 25, 30\}$ over the truncation dimension $\nconsider \in \{50, 100, 200, 300, 500\}$. Compared to the linear dimensionality with encoder and decoder corresponding to the SVD, the quadratic manifold achieves a lower test error \eqref{eq:NumExp:RelError}. The difference between the error obtained with the quadratic manifold and linear dimensionality reduction is lower for this example than for the Hamiltonian wave example from Section~\ref{sec:NumExp:Hamiltonian}. However, when comparing the velocity fields in Figure~\ref{fig:channelflow:reconstruction}, it can be  seen that the quadratic manifold approximates some of the high-frequency features of the velocity field with higher accuracy than linear dimensionality reduction. 
In Figure~\ref{fig:channelflow:singvals}, we study the error in the approximation of the singular values of the online greedy method. Analogously to the results presented for the Hamiltonian wave problem in Section~\ref{sec:NumExp:Hamiltonian}, the approximation accuracy of the singular values increases with the truncation dimension $\nconsider$.

\begin{figure}
    \centering
    \resizebox{1.0\textwidth}{!}{\input{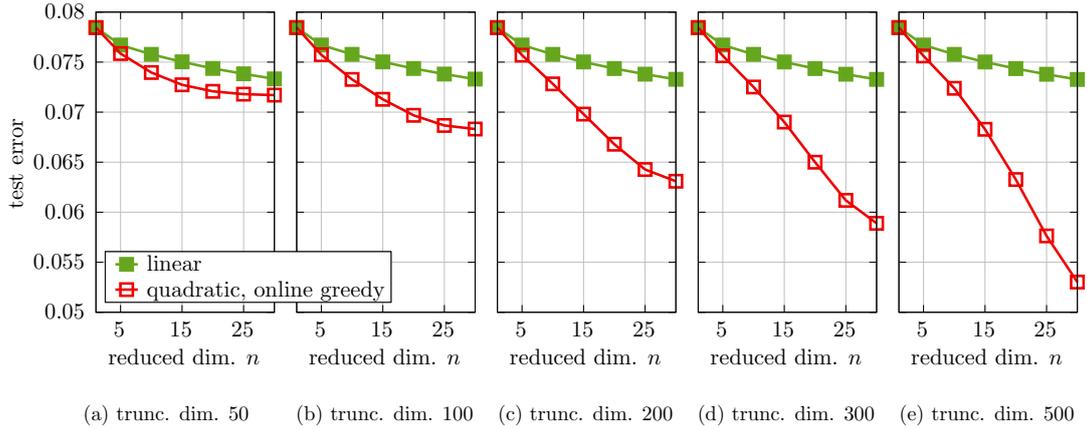}}
    \caption{Channel flow: As the truncation dimension $\nconsider$ is increased, the online greedy method constructs quadratic manifolds from streaming data that achieve about 40\% lower test errors \eqref{eq:NumExp:RelError} than linear dimensionality reduction with the SVD.}
\label{fig:channelflow:test_errs}
\end{figure}

\begin{figure}
    \centering
    \resizebox{1.0\textwidth}{!}{\input{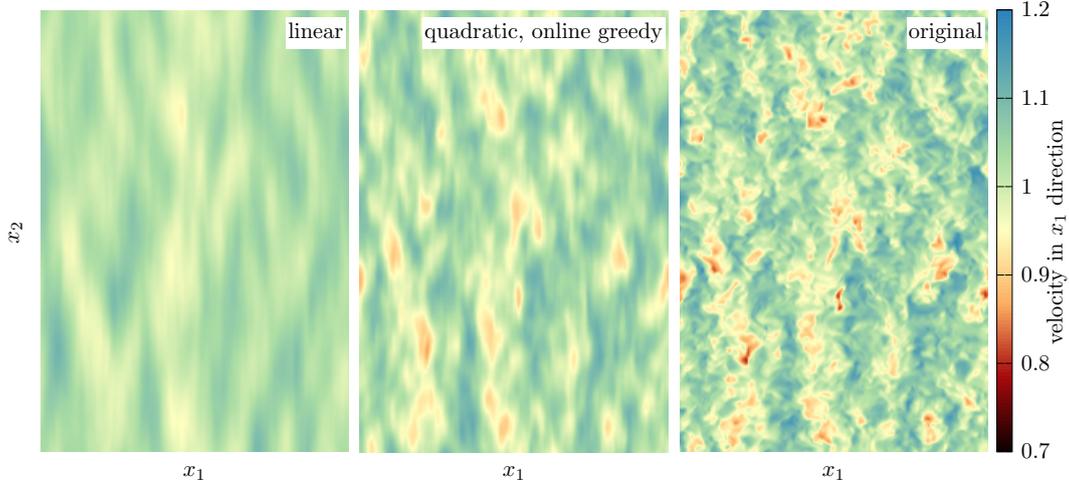}}
    \caption{Channel flow: The online greedy method learns quadratic manifolds that capture fine-scale structure of the turbulent field that are missed by linear dimensionality reduction.}%
    \label{fig:channelflow:reconstruction}
\end{figure}

\begin{figure}
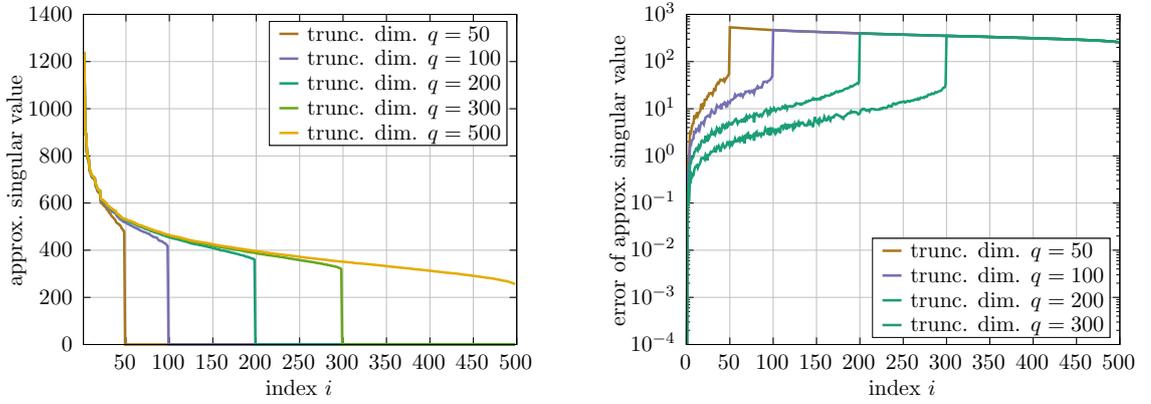

    \centering
    \begin{tabular}{cc}
    \resizebox{0.49\textwidth}{!}{\input{PlotSources/channelflow/channelflow_singvals.tex}}&
    \resizebox{0.49\textwidth}{!}{\input{PlotSources/channelflow/channelflow_singvals_errs.tex}}\\
    \scriptsize (a) approximated singular values given by incremental SVD &
    \scriptsize (b) absolute error of approximated singular values \\
    \end{tabular}
    \caption{Channel flow: As the truncation dimension $\nconsider$ is increased, the approximation accuracy of the singular values increases, which means that the online greedy method can leverage higher-order left-singular vectors for constructing more accurate quadratic manifolds; see also Figures~\ref{fig:channelflow:test_errs}--\ref{fig:channelflow:reconstruction}.}
    \label{fig:channelflow:singvals}
\end{figure}

\subsection{Petabyte-scale data of Kelvin-Helmholtz instability}

\begin{figure}
    \centering
    \resizebox{1.0\textwidth}{!}{\input{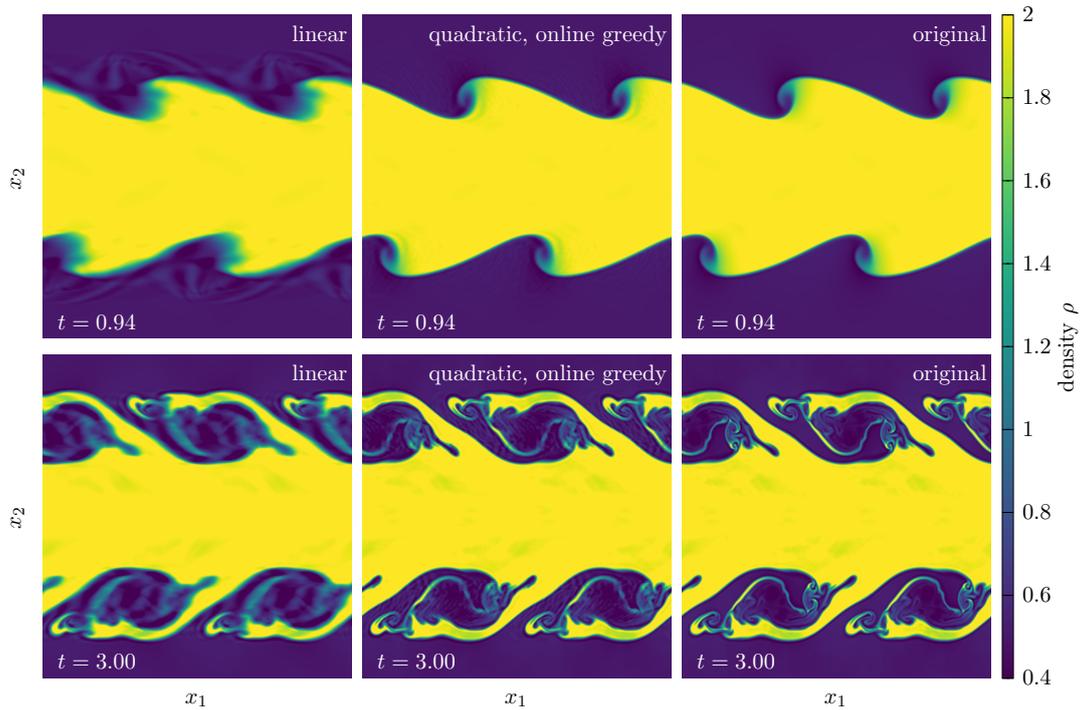}}
    \caption{Kelvin-Helmholtz instability: The online greedy method processes more than one Petabyte of streaming data for constructing a quadratic manifold of the Kelvin-Helmholtz instability in this example. The online greedy method is run alongside the numerical simulation that produces the data in an in-situ fashion so that only small data chunks have to be kept in the main memory and costly disk input/output operations are avoided.} %
    \label{fig:khi:reconstruction}
\end{figure}

\begin{figure}
    \centering
    \resizebox{1.0\textwidth}{!}{\input{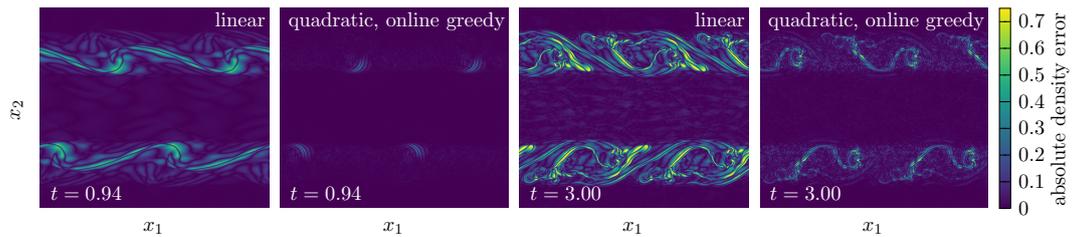}}
    \caption{Kelvin-Helmholtz instability: The quadratic manifold obtained from more than one Petabyte of streaming data achieves lower errors than linear dimensionality reduction, which is visible in the fine-scale structure as the instability emerges. Results are shown for a quadratic manifold of dimension $\nred = 20$.}
    \label{fig:khi:reconstruction_error}
\end{figure}

\begin{figure}
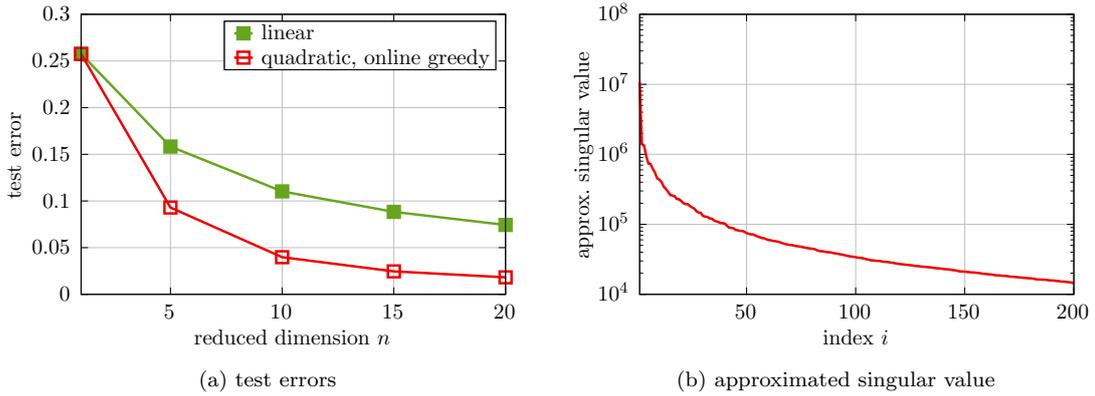

    \centering
    \begin{tabular}{cc}
    \resizebox{0.48\textwidth}{!}{\input{PlotSources/khi/khi_test_errs.tex}} &
    \resizebox{0.48\textwidth}{!}{\input{PlotSources/khi/khi_singvals.tex}} \\
    \scriptsize (a) test errors &
    \scriptsize (b) approximated singular value \\
    \end{tabular}
    \caption{Kelvin-Helmholtz instability: The plots report the error \eqref{eq:NumExp:RelError} of approximating test trajectories on the quadratic manifold learn from the Kelvin-Helmholtz instability. While the test error \eqref{eq:NumExp:RelError} corresponding to the quadratic manifold is only slightly lower than the error corresponding to linear dimensionality reduction, fine-scale features are more accurately captures by the quadratic manifold as can be seen in Figure~\ref{fig:khi:reconstruction_error}.} 
    \label{fig:KHInstabilityTestError}
\end{figure}

We now use the online greedy method as an in-situ data analysis tool to construct quadratic manifolds from streaming data as they are generated by a numerical simulation. This avoids having to store the data on disk because they are processed as they are generated.

Let us consider the Euler equations for gas dynamics in a two-dimensional domain $\Omega = [-1,1)^2$ with periodic boundary conditions, 
\newcommand{\spt}{}
    \begin{align}
    \frac{\partial}{\partial t}
    \begin{bmatrix}
    \rho\spt \\
    \rho\spt v_{x_1}\spt \\
    \rho\spt v_{x_2}\spt \\
    E\spt
    \end{bmatrix}
    +
    \frac{\partial}{\partial x_1}
    \begin{bmatrix}
    \rho\spt v_{x_1}\spt \\
    \rho\spt v_{x_1}\spt^2 + p\spt \\
    \rho\spt v_{x_1}\spt v_{x_2}\spt \\
    v_{x_1}\spt(E\spt+p\spt) \\
    \end{bmatrix}
    +
    \frac{\partial}{\partial x_2}
    \begin{bmatrix}
    \rho\spt v_{x_2}\spt \\
    \rho\spt v_{x_1}\spt v_{x_2}\spt \\
    \rho\spt v_{x_2}\spt^2 + p\spt \\
    v_{x_2}\spt(E\spt+p\spt) \\
    \end{bmatrix}
    =0,
    \end{align}
where $\rho: [0, T] \times \Omega \to \mathbb{R}$, $v_{x_1}: [0, T] \times \Omega \to \mathbb{R}$, $v_{x_2}: [0, T] \times \Omega \to \mathbb{R}$, $E: [0, T] \times \Omega \to \mathbb{R}$, and $p: [0, T] \times \Omega \to \mathbb{R}$ denote the density, velocities in $x_1$ and $x_2$ direction, energy, and pressure, respectively. The equations are closed by the ideal gas law
\begin{equation}\renewcommand{\spt}{(t,\bfx)}
    p\spt=(\Gamma-1)\left(E\spt-\frac{1}{2}\rho\spt (v_{x_1}\spt^2+v_{x_2}\spt^2)\right)\,,\quad \forall \bfx \in \Omega, t \in [0, T]\,,
\end{equation}
with the gas-constant $\Gamma=1.4$. To trigger the Kelvin-Helmholtz instability, we initialize the system based on the setup in~\cite{Rueda-RamirezRG2021Subcell} via
    \begin{align}
    \rho(0, x_1, x_2) &= 0.5 + b(x_1,x_2), \\
    v_{x_1}(0, x_1, x_2) &= \alpha (b(x_1,x_2)-1) + 0.5, \\
    v_{x_2}(0, x_1, x_2) &= 0.1\sin(2\pi x_1), \\
    b(x_1, x_2) &= 0.75(\tanh(\beta (x_2 + 0.5))-\tanh(\beta (x_2 - 0.5)),
    \end{align}
    with a fixed $\beta=80$ and we set the initial pressure to $1.0$ in the entire spatial domain.
We vary $\alpha$ to cover different initial velocity profiles, which lead to changes in the resulting vortex shape. 
The equations of motion are discretized on a grid of size $1024 \times 1024$ using a Harten-Lax-van Leer  flux as described in~\cite[Chapter 6.2.3] {Hesthaven2018Numerical}. We simulate the system until the final time $T=3$. %

Our goal is learning a quadratic manifold from the data generated by 700 simulation runs that correspond to varying the parameter $\alpha$ equidistantly in the interval $[0.495, 0.505]$. One data point consists of the vectorized density, velocities in $x_1$ and $x_2$ directions, and the energy, such that the full dimension is $\nfull=4\times 1024^2 = 4,194,304$.
The 700 simulation runs correspond to  $\nsnapshots = 3.19\times 10^7$ data points. Storing all data points for a post-processing data analysis steps would require 1.07 Petabytes of disk storage.
Instead, we use our online greedy method and run it alongside the simulation in an in-situ fashion. %
In this way, we only need to store the current snapshot chunk together with the truncated SVD. %
In our experiment, the truncation dimension is $\maxrank=200$, which reduces the necessary storage from 1.07PB to roughly 54GB. This fits into the main memory of the compute nodes that we use.
The simulation and data-processing is executed on an NVIDIA H-100 GPU with 80GB of memory; for implementation details we refer to our solver at \url{https://github.com/Algopaul/sqm_demo_gpu}.
In Figure~\ref{fig:khi:reconstruction}, we show the reconstruction of the test-trajectory (where $\alpha=0.5$) for a reduced dimension $\nred=20$. The corresponding point-wise absolute errors are shown in Figure~\ref{fig:khi:reconstruction_error} and the test error is shown in Figure~\ref{fig:KHInstabilityTestError}. As the Kelvin-Helmholtz instability is emerging, the approximation obtained with linear dimensionality reduction exhibits substantial artifacts while the online greedy method provides an approximation that can even resolve the finer features of the underlying flow field. We stress that in this example our online greedy method processes more than one Petabyte of training data and that it is unnecessary to store the data as the greedy method is executed in an in-situ fashion alongside the numerical solver.

\section{Conclusion}\label{sec:Conc}
This work underscores the necessity of in-situ data processing for handling the ever-increasing volumes of data generated by complex numerical simulations. Traditional approaches, which rely on storing large datasets on disk for post-processing, are becoming unsustainable due to the limited main memory as well as slow disk input/output operations. The online greedy method introduced in this work constructs accurate quadratic manifold embeddings directly from streaming data. The data points are processed as they are generated, which allows using the online greedy method as an in-situ data analysis tool. The method's scalability is demonstrated by handling Petabyte-scale data. Furthermore, the method builds on standard linear algebra routines, ensuring compatibility and efficiency across various computing architectures.

\section*{Acknowledgement}
This material is based upon work supported by the U.S.~Department of Energy, Office of Science Energy Earthshot Initiative as part of the project "Learning reduced models under extreme data conditions for design and rapid decision-making in complex systems" under Award \#DE-SC0024721, and the U.S.~Department of Energy, Office of Scientific Computing Research, Award \#DE-SC0019334. Furthermore, this work was authored by the National Renewable Energy Laboratory, operated by Alliance for Sustainable Energy, LLC, for the U.S. Department of Energy (DOE) under Contract No. DE-AC36-08GO28308. Funding provided by Department of Energy Office of Science Advanced Scientific Computing Research Program under Work Authorization No. KJ/AL16/24.

\vskip2pc

\bibliographystyle{myplain}

\bibliography{donotchange,references}

\end{document}